\documentclass{amsart}

\usepackage{graphicx}

\newtheorem{theorem}{Theorem}[section]
\newtheorem{proposition}[theorem]{Proposition}
\newtheorem{lemma}[theorem]{Lemma}
\newtheorem{corollary}[theorem]{Corollary}
\newtheorem{claim}{Claim}

\newcommand{\norm}[1]{\Vert #1 \Vert}
\newcommand{\length}[1]{\mathrm{length}_T ( #1 )}
\newcommand{\diam}[1]{\mathrm{Diam} ( #1 )}
\newcommand{\area}[1]{\mathrm{Area} ( #1 )}

\newcommand{\BoundarySlopeSet}{\mathcal{B}_M}

\begin{document}

\title{Euclidean lengths and the Culler-Shalen norms of slopes}

\author{Kazuhiro Ichihara}
\address{Department of Mathematics, College of Humanities and Sciences, Nihon University,
3-25-40 Sakurajosui, Setagaya-ku, Tokyo 156-8550, Japan}
\email{ichihara.kazuhiro@nihon-u.ac.jp}

\thanks{This work was supported by JSPS KAKENHI Grant Number JP22K03301.}

\dedicatory{Dedicated to Professor Thomas W.~Mattman on his 60th birthday.}


\date{\today}



\begin{abstract}
In the study of exceptional Dehn fillings, two functions on slopes, called the Euclidean length on a horotorus and the Culler-Shalen norm, play important roles. 
In this paper, we investigate their relationship and establish two inequalities between them. 
As a byproduct, some bounds on the boundary slope diameter are given. 
\end{abstract}
\maketitle

\section{Introduction} 

Throughout this paper, let $M$ denote a hyperbolic knot manifold,
that is, a compact, orientable 3-manifold with a single torus boundary $\partial M$,
whose interior admits a complete Riemannian metric of constant curvature $-1$ with finite volume.
The main object studied in this paper is the set of slopes on $\partial M$.
By a \textit{slope} on $\partial M$, we mean the isotopy class of an unoriented simple closed curve on $\partial M$.


Let us start with the motivation and background of our research.
An operation that glues a solid torus $V$ to $M$ along their boundaries is called a \textit{Dehn filling} on $M$.
The slope $r$ on $\partial M$, whose representative is identified with the meridian of $V$, determines the homeomorphism type of the resulting manifold.
Thus, we denote the resulting manifold by $M(r)$, and say that the manifold is obtained from $M$ by Dehn filling along the slope $r$.
Thurston’s well-known hyperbolic Dehn surgery theorem \cite[Theorem 5.8.2]{Th} states that all but finitely many Dehn fillings on $M$ produce closed hyperbolic 3-manifolds.
Accordingly, a Dehn filling on $M$ that yields a non-hyperbolic 3-manifold is called \emph{exceptional}.
Many works have been devoted to studying when, how many, and what types of exceptional Dehn fillings can occur.

In the study of exceptional Dehn fillings, two functions on slopes on $\partial M$, called the \textit{Euclidean length} on a horotorus $T$ and the \textit{Culler-Shalen norm} on $H_1(\partial M ; \mathbb{R})$, play important roles.
See \cite{B} for example, for a survey.
In this paper, we use the notations $\length{\cdot}$ and $\norm{\cdot}$ to denote them, respectively.

Here we give very rough definitions of the two functions.
The Euclidean length of a slope on a horotorus $T$ is defined as the minimal length of its representatives on $T$.
A \textit{horotorus} $T$ in $M$ appears as the image of equivariant horospheres in the hyperbolic $3$-space $\mathbb{H}^3$ under the universal covering map.
This $T$ is naturally identified with $\partial M$ and is endowed with a Euclidean metric induced from the hyperbolic metric on $M$.
Thus, the Euclidean length of curves on $\partial M$ can be defined on $T$, and the Euclidean length of a slope on $\partial M$ is defined accordingly. 
On the other hand, the Culler-Shalen norm is defined by using the set of characters of $SL(2,\mathbb{C})$-representations of the fundamental group $\pi_1(M)$.
This set forms a complex affine algebraic set, which is called the \textit{character variety} of $M$ and denoted by $X(M)$. 
For a slope $r$, by using the trace function, a regular map $I_{\boldsymbol{r}}$ onto $\mathbb{CP}^1$ is defined on the smooth projective model of a one-dimensional component of $X(M)$. 
When $M$ is hyperbolic, there exists a discrete faithful representation $\pi_1(M) \to SL(2,\mathbb{C})$. 
The algebraic component $X_0$ of $X(M)$ containing the character of the representation is known to be one-dimensional. 
From twice the degree of $I_{\boldsymbol{r}}$ on $X_0$, a norm $\norm{ \cdot }_0$ on the real vector space $H_1 (\partial M; \mathbb{R})$ can be obtained, which is called the \emph{Culler–Shalen norm associated to $X_0$}. 
The same construction works for other one-dimensional components of $X(M)$, and in general, a seminorm on $H_1 (\partial M; \mathbb{R})$ is obtained.
Then the \emph{(total) Culler-Shalen norm} $\norm{\cdot}$ on $H_1 (\partial M; \mathbb{R})$ is defined as the total sum of the seminorms on the irreducible one-dimensional components of $X(M)$. 
See the next section for their precise definitions.

Concerning these two functions, we have the following excellent results.

\begin{itemize}
\item
If $r$ is not a strict boundary slope and $\pi_1(M(r))$ is cyclic,
then $\norm{r}_0$ takes the least value among non-trivial slopes \cite[Corollary 1.1.4]{CGLS}.
\item
If $r$ is not a strict boundary slope and $\pi_1(M(r))$ is finite,
then $\norm{r}_0$ is at most $5$ times the least value among non-trivial slopes \cite[Theorem 2.1]{BZ}.
\item
If $M(r)$ does not admit a negatively-curved metric,
then $\length{r}$ is at most $2\pi$ \cite[Theorem 9 (The “$2 \pi$” Theorem)]{BH}.
\item
If $M(r)$ is reducible or $\pi_1(M(r))$ is not word-hyperbolic, then $\length{r}$ is at most $6$ \cite{Ag,L}.
\end{itemize}

See Subsection~\ref{subsec:21} for details of the definitions of the terms appearing above. 

Though the definitions of the two functions are quite different,
all the theorems above seem to have somehow similar flavors:
If $M(r)$ is non-hyperbolic, then the value of $r$ is relatively small.
Thus, it seems natural to ask whether there exists some relationship between them.

\bigskip


In this paper, motivated by the question above, we give some inequalities between them.
We first give a lower bound on the (total) Culler-Shalen norm in terms of the Euclidean length on a horotorus when $M$ is a certain knot exterior in the $3$-sphere $S^3$.

\begin{theorem}\label{th-norm>length}
Suppose that $M$ is the exterior of a hyperbolic two-bridge knot or the exterior of a $(-2,3,n)$-pretzel knot with $n$ odd and at least $7$ in $S^3$. 
Then 
$$
\norm{r} \geq \frac{2}{3} \, \length{r}
$$
holds for any slope $r$ on $\partial M$ and for any horotorus $T$. 
In particular, if $M$ is the exterior of a hyperbolic twist knot or the exterior of a $(-2,3,n)$-pretzel knot with $n$ odd and at least $7$ in $S^3$, then the same inequality holds for $\norm{r}_0$. 
\end{theorem}

A knot $K$ in $S^3$ is called a \textit{two-bridge} knot if its bridge index (the minimum number of local maxima (or local minima) with respect to an axis) is two.
Since two-bridge knots are alternating, a natural consequence of Menasco’s work in \cite{Me84} is that a two-bridge knot is hyperbolic unless it is a $(2,p)$-torus knot. 
Any two-bridge knot is represented by the well-known Conway diagram $C(a_1, \dots,a_m)$ with non-negative integers $a_i$. 
A two-bridge knot represented by the diagram $C(2,n)$ is called a \emph{twist knot}. 
See \cite{IMS} for example. 
A knot $K$ in $S^3$ presented by a diagram obtained by putting rational tangles of the form $1/q_1, \dots, 1/q_n$ in line is called a \textit{pretzel knot}, denoted by $P(q_1, \dots, q_n)$, with integers $q_1, \dots, q_n$.
It is known that a pretzel knot $P(-2,3,n)$ is hyperbolic unless $n = 1, 3, 5$. 
See \cite{Ma} for example. 

Our proof depends on the existence of a pair of boundary slopes with some particular property, which we can verify for such two-bridge knots and $(-2,3,n)$-pretzel knots.
In fact, it is possible that the above inequality holds in a much wider class. 
For example, the inequality also holds for hyperbolic knots
with only two strict essential surfaces, both of which are spanning surfaces, in a general $3$-manifold (Proposition~\ref{prp-TwoSurfaceKnot}). 
See Section~\ref{sec-Norm>Length} for more details.


On the other hand, in general, the converse inequality does not hold.
Indeed, we can find an example of a $3$-manifold
whose boundary contains a slope $r$ with
arbitrarily small ratio $\length{r} / \norm{r}$ for any horotorus $T$ (Proposition~\ref{prp-nonUnivBounds}).
Taking this into account, we obtain the following weaker inequality.
Here, a slope on $\partial M$ is called \textit{integral}
if it has distance one from the meridian, and
called a \textit{boundary slope} if
there exists an embedded essential surface in $M$
whose boundary component represents the slope.
For more details on these terminologies, see the next section.

\begin{theorem}\label{th-length>norm}
Let $m$ be the fixed meridian,
and let $r_1, r_2$ be integral slopes on $\partial M$.
If $r_1$ is greater than or equal to the maximal boundary slope, and $r_2$ is less than or equal to the minimal boundary slope for $M$, then
$$
\length{r_1} + \length{r_2} > \frac{\norm{r_1}+\norm{r_2}}{\norm{m}}
$$
holds for the maximal horotorus $T$. 
\end{theorem}

Precisely, we will prove an extended version of this theorem,
in which the inequality includes the data of the denominators of the numerical slopes.
In the proof, the difference between two numerical slopes plays an important role.
It will be presented in Section~\ref{sec-Length>Norm}.

\bigskip


As a byproduct of the method to prove the theorem above,
we have an improvement on
a lower bound for the boundary slope diameter obtained in \cite{IMS}.

When we set a meridian-longitude system on $\partial M$,
the set of slopes on $\partial M$ is identified
with the set of rational numbers together with $1/0$.
See \cite{R}, for example.
Regarding the set of non-meridional boundary slopes, denoted by $\BoundarySlopeSet$,
as a subset of the set $\mathbb{Q}$ of rational numbers,
let $\diam{\BoundarySlopeSet}$ be the diameter of the set,
i.e., the difference between the greatest and the least elements.

In \cite{CS99}, it is shown that $\diam{\BoundarySlopeSet} \geq 2$ if $M$ is the exterior of a non-trivial, non-cabled knot in an orientable $3$-manifold with cyclic fundamental group.
In terms of the Culler-Shalen norm of a boundary slope associated to an ideal point, a generalization of such a lower bound on $\diam{\BoundarySlopeSet}$ was obtained by Ishikawa, Mattman, and Shimokawa \cite{IMS}.
The following gives an extension of their result.

\begin{theorem}\label{th-diam}
Let $r=p/q$ be a non-meridional numerical boundary slope on $\partial M$. 
Then
$$
\diam{\BoundarySlopeSet} > \frac{\norm{r}}{q \norm{m}}
$$
holds, where $m$ denotes the meridional slope on $\partial M$. 
\end{theorem}

Similar arguments also give similar upper bounds on $\diam{\BoundarySlopeSet}$.
The proof of Theorem~\ref{th-diam} will be presented in Section~\ref{sec-diam}.


In the last section, as an example, we investigate for the figure eight knot exterior the relationship between the Euclidean length on the maximal horotorus and the Culler-Shalen norm of slopes.

\section*{Acknowledgments} 

The author would like to thank Masaharu Ishikawa and Thomas Mattman for helpful discussions, and also thank to Shigeru Mizushima for his help on calculations about boundary slopes for $(-2,3,n)$-pretzel knots. 
He also thanks the anonymous referee for useful comments.

\section{Definitions} 

In this section, we prepare the definitions used in this paper.
Recall that $M$ always denotes a connected, compact, orientable $3$-manifold with a single torus boundary $\partial M$, whose interior is homeomorphic to a complete hyperbolic manifold with finite volume.


\subsection{Slopes}\label{subsec:21}

As usual, the isotopy class of an unoriented, non-trivial simple closed curve on $\partial M$ is called a \textit{slope}.
When we fix a particular simple closed curve on $\partial M$, called the \textit{meridian} (respectively, the \textit{longitude}), we refer to the slope $m$ represented by the meridian as the \textit{meridional slope} (resp. \textit{longitudinal slope}).
The \textit{distance} between two slopes is defined as
the minimal geometric intersection number of their representatives.
We use $\Delta(r, r')$ to denote the distance between slopes $r$ and $r'$.
If we take a pair of a meridian and a longitude on $\partial M$, we always assume that the distance between them is one.
Such a pair is called a \textit{meridian-longitude system} on $\partial M$.
When a meridian–longitude system, say ${ m, l }$, on $\partial M$ is fixed, the set of slopes is identified with the set of rational numbers together with $1/0$:
a slope $r$ corresponds to the fraction $\Delta(r, l) / \Delta(r, m)$.
This gives a natural bijection. See \cite{R}, for example.
We will often abuse the terminology and identify a slope with its corresponding rational number, calling it a \textit{numerical slope} when necessary.
From this point of view, a slope is called \textit{integral} if the corresponding numerical slope lies in $\mathbb{Z}$, i.e., an integral slope has distance one from the meridional slope.
We observe that, for a pair of numerical slopes $p/q$ and $r/s$, the distance $\Delta(p/q, r/s)$ and the difference $|p/q - r/s|$ as rational numbers are related as
\[
\Delta(p/q,r/s) = | ps - qr | = q s \left| \frac{p}{q} - \frac{r}{s} \right|.
\]


The isotopy classes of oriented, non-trivial simple closed curves on $\partial M$ are in natural bijective correspondence with the primitive elements of $H_1(\partial M; \mathbb{Z})$.
Here, an element of $H_1(\partial M; \mathbb{Z})$ is called \textit{primitive} if it is not an integral multiple of any other non-trivial element.
Note that two primitive elements $\boldsymbol{r}, \boldsymbol{r}' \in H_1(\partial M; \mathbb{Z})$
correspond to the same slope if and only if $\boldsymbol{r} = \pm \boldsymbol{r}'$.
Thus, forgetting the orientations, there is a natural two-to-one map from the set of primitive elements of $H_1(\partial M; \mathbb{Z})$ onto the set of slopes on $\partial M$.
Since a meridian and a longitude have distance one, a meridian–longitude system with orientations gives a set of generators of $H_1(\partial M; \mathbb{Z})$,
and also of $\pi_1(\partial M)$.
In particular, a numerical slope $p/q$ corresponds to
the elements $\pm (p \, \boldsymbol{m} + q \, \boldsymbol{l}) \in H_1(\partial M; \mathbb{Z})$,
where $\boldsymbol{m}$ and $\boldsymbol{l}$ denote the elements
corresponding to the oriented meridian and longitude, respectively.


One more key object in the proofs of our results is the boundary slope of an essential surface in $M$.
A two-sided, possibly disconnected, properly embedded surface $F$ in $M$ is called \textit{essential} if it is non-empty, $\pi_1$-injective, and has no $2$-sphere components and no boundary-parallel components.
The boundary components of an essential surface $F$ consist of a parallel family of simple closed curves on $\partial M$.
Thus, they determine a slope, which is called the \textit{boundary slope} of $F$.
A connected essential surface in $M$ is called a \textit{semi-fiber} if either $F$ is a fiber in a fibration of $M$ over $S^1$, or $F$ is the common frontier of two $3$-dimensional submanifolds of $M$,
each of which is a twisted $I$-bundle with associated $\partial I$-bundle $F$.
An essential surface $F$ in $M$ is termed
\textit{strict} if no component of $F$ is a semi-fiber.


\subsection{Length on a horotorus}

In this subsection, we define the Euclidean length of a slope, the angle between slopes, and collect some lemmas used in the following sections.


We are supposing that the interior of $M$, denoted by $\mathrm{Int}(M)$,
admits a complete hyperbolic structure of finite volume.
This means that there exists
a universal covering map $p: \mathbb{H}^3 \to \mathrm{Int} (M)$,
which is a local isometry.
The covering transformation group,
which is isomorphic to the fundamental group $\pi_1(M)$,
acts on $\mathbb{H}^3$ properly discontinuously.
For this action, take an equivariant set of horospheres,
which bound horoballs with disjoint interiors in $\mathbb{H}^3$.
These descend to a properly embedded torus in $M$, which we call a \textit{horotorus}.
A horotorus $T$, together with $\partial M$, cobounds a region homeomorphic to $T^2 \times [0,1]$,
since $\pi_1(\partial M) \cong \mathbb{Z} \oplus \mathbb{Z}$
acts properly discontinuously on the horoball bounded by a horosphere
that appears as a component of $p^{-1}(T)$.
Thus, there is a natural identification between $T$ and $\partial M$.
Such a horotorus $T$ in $M$ is regarded as
a Euclidean torus, as demonstrated in \cite{Th}.
That is, the Riemannian metric induced from
the restriction of the hyperbolic metric of $M$ onto $T$ is Euclidean.
By using this metric,
the length of a curve on $\partial M$ can be defined.
The \textit{length} of a slope $r$ on $\partial M$ is defined as
the minimum of the lengths of
simple closed curves with slope $r$, and
we denote it by $\length{r}$.
Note that this length depends upon the choice of $T$.


When an oriented meridian–longitude system is fixed on a horotorus $T$, we define the \textit{angle} between two slopes on $T$ as follows.
Let $r_1, r_2$ be a pair of slopes on $T$.
Take Euclidean geodesics on $T$ representing $r_1$ and $r_2$.
Orient them so that they are coherent with the longitude.
Then the angle between $r_1$ and $r_2$ is defined as
the angle between the geodesics so oriented, measured on $T$.
We measure the angle from $0$ to $\pi$.
In the following, for a slope $r$, we denote by $\theta_r$ the angle between $r$ and the meridional slope, with $0 < \theta_r \leq \pi$.

With these settings, the next lemma is well-known. 
See \cite{Ag} for example. 

\begin{lemma}\label{lem-LengthDistance}
For a pair of slopes $r, r'$, 
$$
\Delta(r, r') = \frac{ \length{r} \cdot \length{r'} \cdot \sin | \theta_r - \theta_{r'} | }{\area{T}}
$$
holds, where $\area{T}$ denotes the Euclidean area of the horotorus $T$. 
\end{lemma}

Here, for completeness, we include a simple proof of this lemma by using a graphical view of the length and angle for slopes.
This will also be used in Subsection \ref{subsec-Length>Norm-Length}.
As stated, we identify $\partial M$ with a horotorus $T$, and have fixed a base point on $T$.
Let $\widetilde{T}$ be a component of the preimage of $T$ in the universal covering space $\mathbb{H}^3$ of $\mathrm{Int}(M)$.
This $\widetilde{T}$ is a horosphere in $\mathbb{H}^3$, and so it is naturally identified with the Euclidean plane.
Note that $\widetilde{T}$ is also regarded as the universal covering space of $T$, and they are locally isometric.
Under the covering projection, the preimage of the base point on $T$ gives a lattice on $\widetilde{T}$.
We then take a base point $O$ in the lattice.
A vector from $O$ to a primitive lattice point projects down to an oriented simple closed geodesic on $T$.
Since such a geodesic uniquely exists in the homology class, it determines a primitive element of $H_1(T; \mathbb{Z})$.
Together with the relation between the slopes on $T$ and $H_1(T; \mathbb{Z})$,
we obtain a one-to-two correspondence between the slopes on $T$ and the vectors on $\widetilde{T}$ emerging at $O$ and ending at the primitive lattice points:
Two such vectors on $\widetilde{T}$ correspond to the same slope if and only if their sum is the zero vector.
Now we suppose that an oriented meridian–longitude system on $\partial M$ is arbitrarily chosen and fixed.
Recall that such a system gives a set of generators of $H_1(\partial M; \mathbb{Z})$.
Using the identification of $T$ and $\partial M$, we can take two oriented Euclidean geodesics of meridional and longitudinal slope on $T$ such that both run through the base point on $T$.
We set a $u$–$v$ coordinate on $\widetilde{T}$ so that
the vectors appearing as lifts of those two oriented geodesics starting from the origin end at $(0,1)$ and $(1,0)$, respectively.
For this coordinate, by linearity, the vector appearing as the lift of the oriented geodesic on $T$ with numerical slope $p/q$ (respectively $-p/q$) corresponds to the vector ending at the point $(q,p)$ (resp. $(q,-p)$) on $\widetilde{T}$.
Here we give an orientation to the geodesic on $T$ so that it is coherent with that for the longitudinal slope.
In this way, we have a one-to-one correspondence between the non-meridional slopes on $\partial M$ and the vectors ending at the primitive lattice points in the half-plane $\{ (u,v) \mid u > 0 \}$ on $\widetilde{T}$.
Note that the meridional slope corresponds to the vector ending at $(0,1)$ on $\widetilde{T}$.
In this view, the length of a slope on $T$ is just the length of the corresponding vector, and the angle between two slopes is just the angle between the corresponding two vectors on $\widetilde{T}$.

\begin{proof}[Proof of Lemma \ref{lem-LengthDistance}]
For a pair of numerical slopes $r = p/q$, $r' = p'/q'$,
$\Delta(r, r') = |p q' - p' q|$ holds.
See \cite{B}, for example.

Now we consider the parallelogram $P$ on $\widetilde{T}$
spanned by the vectors corresponding to $r$ and $r'$.
Then the area of $P$ is expressed by
$|p q' - p' q|$ times the area of $T$.
The area of the parallelogram
spanned by the vectors corresponding
to the meridional slope and the longitudinal slope is $\area{T}$.

On the other hand, the area of $P$ is calculated by
$\length{r} \cdot \length{r'} \cdot \sin \theta$,
where
$\theta$ denotes the angle between the vectors corresponding to $r$ and $r'$.

Since $\theta$ is equal to $|\theta_r - \theta_{r'}|$, we have
\begin{eqnarray*}
\Delta(r, r') &=& | p q' - p' q| 
= \frac{\length{r} \cdot \length{r'} \cdot \sin \theta}{\area{T}} \\
&=&\frac{\length{r} \cdot \length{r'} \cdot \sin | \theta_r - \theta_{r'} |}{\area{T}}\ .
\end{eqnarray*}
\end{proof}

We also prepare the following two lemmas, which will be used in the next section.
Here, a horotorus $T$ is called \textit{maximal} if it is maximal among those each of which has no overlapping interior.
If we need an embedded one, we shrink it by an arbitrarily small amount.

\begin{lemma}[Adams {\cite[Lemma 2.4]{Ad}}]\label{lem-minimallength}
Assume that a horotorus $T$ in $M$ is maximal. 
Then $\length{r} \geq 1$ holds for any slope $r$ on $\partial M$. 
\end{lemma}

\begin{lemma}[Agol {\cite[Lemma 5.1]{Ag}} ]\label{lem-LengthEulerChara}
Suppose that an essential surface $S$ with boundary slope $s$ in $M$ is given. 
Then 
$$
\length{s} \leq 6 \frac{ -\chi}{b}
$$ 
holds, where $\chi$ denotes the Euler characteristic of $S$ and $b$ the number of boundary components of $S$.
\end{lemma}

These are simplified versions of the results obtained in \cite{Ad} and \cite{Ag}, so we omit the proofs here.

\subsection{Culler-Shalen norm}

In this subsection, we define the Culler--Shalen norm of a slope and state a lemma used in the next section.
The theory of such norms was originally developed by Culler and Shalen.
See \cite{CS04} as a basic reference for this subsection.
See also \cite{S} for a survey.

We denote by $R(\pi_1(M))$ the set of representations of $\pi_1(M)$ into $SL_2(\mathbb{C})$.
It is naturally regarded as a complex affine algebraic set.
We denote by $X(M)$ the set of all characters of representations in $R(\pi_1(M))$.
The set $X(M)$ is endowed with the structure of an affine algebraic set as in \cite{CGLS}.
Thus it is called the \textit{character variety} of $\pi_1(M)$, or simply of $M$.
For a hyperbolic 3-manifold $M$, there exists a particular $SL_2(\mathbb{C})$-representation of $\pi_1(M)$ induced from the hyperbolic structure.
This is a discrete, faithful representation of $\pi_1(M)$ into $SL_2(\mathbb{C})$, the so-called holonomy representation.
See \cite{Th} for a basic reference.
Then we define a \emph{principal component} $X_0$ of the character variety $X(M)$ to be a component that contains the character of the holonomy representation. 
This $X_0$ is known to be an irreducible one-dimensional component of $X(M)$, i.e., a complex affine algebraic curve, due to Thurston. 
See \cite{S} for detailed explanations.
For $X_0$, we denote by $\tilde{X_0}$ the unique smooth projective curve that admits a birational correspondence $\phi: \tilde{X_0} \to X_0$. 
That is, the curve $\tilde{X_0}$ is constructed by desingularization to obtain a projective completion of $X_0$.
A point $x \in \tilde{X_0}$ is said to be an \textit{ideal point} if it does not correspond to any point of $X_0$ under $\phi$.
If $\gamma$ is an element of $\pi_1(M)$, then let $I_\gamma : X_0 \to \mathbb{C}$ denote the rational function on $X_0$ defined by $I_\gamma(\chi) = \chi(\gamma) = \operatorname{trace}(\rho(\gamma))$.

The next proposition assures the existence of a norm on $H_1(\partial M; \mathbb{R})$, which is called the \textit{Culler-Shalen norm} associated to $X_0$.

\begin{proposition}[Culler--Shalen {\cite[Proposition 5.7]{CS04}}]\label{prop24}
Let $X_0$ be a principal component of $X(M)$ and let $x_1, \cdots, x_n$ denote the ideal points of $\tilde{X_0}$.
Then there exists a unique norm $\norm{\cdot}_0$ on the vector space $H_1(\partial M; \mathbb{R})$ such that for any element $\alpha \in H_1(\partial M; \mathbb{Z}) \subset H_1(\partial M; \mathbb{R})$, it satisfies 
$$\norm{\alpha}_0 = \deg\!\left( {I_{[c]} |_{X_0}}^2 - 4 \right) = 2\,\deg\!\left( I_{[c]} |_{X_0} \right).$$ 
Here $[c]$ denotes the element in $\pi_1(\partial M)$ represented by a closed curve $c$ corresponding to $\alpha$.
Moreover, there are strict essential surfaces $S_1, \cdots, S_n$ in $M$ corresponding to the ideal points $x_1, \cdots, x_n$ of $\tilde{X_0}$ such that
\begin{equation}\label{eq-CS0}
\norm{\alpha}_0 = \sum_{j=1}^n \Delta_{\partial M}(c, \partial S_j)
\end{equation}
holds for any homology class $\alpha \in H_1(\partial M; \mathbb{Z}) \subset H_1(\partial M; \mathbb{R})$ represented by a closed curve $c$ in $\partial M$.
Here $\Delta_{\partial M}(c, \partial S_j)$ denotes the minimal geometric intersection number on $\partial M$ between the closed curve $c$ and $\partial S_j$, which we regard as a parallel family of closed curves. 
\end{proposition}

We remark that each $S_j$ in the proposition above is the surface captured by the action on the tree $T_{x_j}$ associated to the ideal point $x_j$, which is called a $T_{x_j}$-surface in \cite{CS04}. 
In fact, it is a possibly disconnected, separating surface properly embedded in $M$. 
See \cite{CS04} for further details.

The same machinery as above almost works for any one-dimensional component of the character variety $X(M)$ that contains the character of an irreducible representation, and one can, in general, obtain a seminorm associated with the component. 
Then, the (total) Culler-Shalen norm $\norm{\cdot}$ on $H_1(\partial M; \mathbb{R})$ is defined as the total sum of the values of the seminorms on the irreducible one-dimensional components of $X(M)$.

 If the seminorm obtained for a one-dimensional component (curve) of $X(M)$ is a genuine norm, then the component is called a \emph{norm curve}. 
By Proposition~\ref{prop24}, the principal component $X_0$ is a norm curve.

Based on these settings, we have the following lemma, which is well known to specialists.

\begin{lemma}\label{lem-NormDistance}
Let $s_1, \dots, s_m$ be all the boundary slopes on $\partial M$. 
For any element $\boldsymbol{\alpha} \in H_1(\partial M; \mathbb{Z}) \subset H_1(\partial M; \mathbb{R})$,
\begin{equation}\label{eq-CS}
\norm{\boldsymbol{\alpha}} = \sum_{i=1}^m a_i \, \Delta(r, s_i)
\end{equation}
holds, where each $a_i$ is a non-negative integer, and $r$ is the slope corresponding to $\boldsymbol{\alpha} \in H_1(\partial M; \mathbb{Z}) \subset H_1(\partial M; \mathbb{R})$. 
Moreover, if $s_i$ is associated to an ideal point of some norm curve of $X(M)$, then $a_i$ is a positive even integer.
\end{lemma}

\begin{proof}
Consider first the Culler-Shalen norm $\norm{\cdot}_0$ associated with a principal component $X_0$. 
From Equation~\eqref{eq-CS0}, the equality for $\norm{\cdot}_0$ in the lemma is obtained by setting $a_i = 0$ whenever $s_i$ does not correspond to a surface $S_j$ as in the proposition above. 
Also, by \cite[Proposition 3.10]{CS04}, each $S_j$ in the proposition is a separating surface in $M$.
This implies that, by setting $a_i$ as the sum of the numbers of connected components of the $\partial S_j$'s, the sum of $\Delta_{\partial M} ( c , \partial S_j )$'s for the surfaces $S_j$ with boundary slope $s_i$ is equal to $a_i \, \Delta ( r , s_i )$. 
Since $\Delta_{\partial M} ( c , \partial S_j )$ denotes the minimal geometric intersection number between $c$ and $\partial S_j$, and $S_j$ is a (possibly disconnected) separating surface, each $a_i$ must be a positive even integer. 
For the total Culler-Shalen norm $\norm{\cdot}$, by definition, the same arguments imply the statements in the lemma.
\end{proof}

As explained above, there is a natural two-to-one map
from the set of primitive elements of $H_1(\partial M; \mathbb{Z})$ onto the set of slopes on $\partial M$.
That is, a slope $r$ corresponds to $\pm \boldsymbol{r} \in H_1(\partial M; \mathbb{Z})$.
Since $\norm{\boldsymbol{r}} = \norm{-\boldsymbol{r}}$ holds, we define $\norm{r}$ for a slope $r$ by setting $\norm{r} = \norm{\boldsymbol{r}} = \norm{-\boldsymbol{r}}$.
In the rest of the paper, we will use this notation.

\section{Lower bound on Culler-Shalen norm}\label{sec-Norm>Length}

In this section, we give a proof of Theorem~\ref{th-norm>length}.
The essential part of the proof is provided by the following proposition.

\begin{proposition}\label{prop-Norm>Length}
Suppose that there exist two distinct boundary slopes $s_1, s_2$ of essential surfaces $S_1, S_2$ in $M$ such that they are associated with ideal points of some norm curves in $X(M)$. 
Let $\chi_i$ denote the Euler characteristic of $S_i$ and $b_i$ the number of boundary components of $S_i$ for $i=1,2$, respectively. 
If the inequality 
$$  
\Delta(s_1 , s_2) \geq 2 \frac{-\chi_i}{ b_i},  
$$  
is satisfied, then  
$$  
\norm{r} \geq  
\frac{2}{3} \, \length{r}  
$$  
holds for any slope $r$ on $\partial M$ and for any horotorus $T$.
\end{proposition}

\begin{proof}
Let $s_1, \cdots , s_n$ be boundary slopes on $\partial M$. 
By Lemma \ref{lem-NormDistance}, we have 
$
\norm{ r }  = \sum^n_{i=1} a_i \Delta ( r, s_i)
$
for a slope $r$, with non-negative integers $a_1, \cdots, a_n$ with $a_i \ge 2$. 

By assumption, there exist two boundary slopes, say $s_1, s_2$, on $\partial M$ associated with some ideal points of norm curves in $X(M)$. 
This implies that $a_1$ and $a_2$ are both strictly positive, in fact, they are both at least $2$, by Lemma \ref{lem-NormDistance}. 

Then, obviously we have
$$
\norm{ r }  \geq a_1 \Delta ( r, s_1) + a_2 \Delta ( r, s_2) > 0 \ .
$$

Together with Lemma \ref{lem-LengthDistance}, this implies  
\begin{eqnarray*}
\norm{ r }  
&\geq& 
\frac{\length{r}}{\area{T}}
\left(
a_1 \length{s_1} \sin | \theta_r - \theta_{s_1} | 
+ 
a_2 \length{s_2} \sin | \theta_r - \theta_{s_2} | 
\right) \ .
\end{eqnarray*}

Now, without loss of generality, we assume that $a_1 \length{s_1} \geq a_2 \length{s_2}$. 
Then we obtain:
\begin{equation}\label{eq-1}
\norm{ r }  
\geq
\frac{\length{r}}{\area{T}} \cdot a_2 \length{s_2} \cdot 
\left(
\sin | \theta_r - \theta_{s_1} | + \sin | \theta_r - \theta_{s_2} | 
\right)
\end{equation}

\begin{claim}
$
\sin | \theta_r - \theta_{s_1} | + \sin | \theta_r - \theta_{s_2} | 
\ge 
\sin | \theta_{s_1} - \theta_{s_2} | 
$
holds. 
\end{claim}

\begin{proof}
First we assume that $\theta_{s_1} > \theta_{s_2}$. 
Then we have three cases. 
If $\pi > \theta_r \geq \theta_{s_1} > \theta_{s_2} > 0$, then the following holds: 
\begin{eqnarray*}
\sin | \theta_{s_1} - \theta_{s_2} | 
&=&
\sin ( \theta_{s_1} - \theta_{s_2} ) \\
&=&
\sin \left( ( \theta_r - \theta_{s_2} ) - ( \theta_r - \theta_{s_1} ) \right) \\
&=&
\sin ( \theta_r - \theta_{s_2} ) 
\cos ( \theta_r - \theta_{s_1} ) 
-
\cos ( \theta_r - \theta_{s_2} ) 
\sin ( \theta_r - \theta_{s_1} ) \\
& \le &
\sin ( \theta_r - \theta_{s_2} ) + \sin ( \theta_r - \theta_{s_1} ) \\
& = & 
\sin | \theta_r - \theta_{s_1} | + \sin | \theta_r - \theta_{s_2} | 
\end{eqnarray*}
For the case where $\pi > \theta_{s_1} > \theta_{s_2} \geq \theta_r \geq 0$, 
similar calculations can apply. 
If $\pi > \theta_{s_1} \geq \theta_r \geq \theta_{s_2} > 0$, then the following also holds: 
\begin{eqnarray*}
\sin | \theta_{s_1} - \theta_{s_2} | 
&=&
\sin ( \theta_{s_1} - \theta_{s_2} ) \\
&=&
\sin \left( ( \theta_{s_1} - \theta_r ) + ( \theta_r - \theta_{s_2} ) \right) \\
&=&
\sin ( \theta_{s_1} - \theta_r ) 
\cos ( \theta_r - \theta_{s_2} ) 
+
\cos ( \theta_{s_1} - \theta_r ) 
\sin ( \theta_r - \theta_{s_2} )  \\
& \le  &
\sin ( \theta_{s_1} - \theta_r )  + \sin ( \theta_r - \theta_{s_2} ) \\
& = & 
\sin | \theta_r - \theta_{s_1} | + \sin | \theta_r - \theta_{s_2} | 
\end{eqnarray*}
It can be shown for the case where $\theta_{s_1} < \theta_{s_2}$ in the same way. 
\end{proof}

It follows from this claim and Equation \ref{eq-1} 
$$
\norm{ r }  
\geq
\frac{\length{r}}{\area{T}} \cdot a_2 \length{s_2} \cdot 
\sin | \theta_{s_1} - \theta_{s_2} | \ .
$$

Again, by Lemma \ref{lem-LengthDistance}, we get 
\begin{eqnarray*}
\norm{ r }  
& \geq &
\frac{\length{r} \cdot a_2 \cdot \length{s_2} \cdot \sin | \theta_{s_1} - \theta_{s_2} |}{\area{T}} \\
& = &
a_2 \cdot \frac{\length{r}}{\length{s_1}} \cdot 
\frac{\length{s_1} \length{s_2} \sin | \theta_{s_1} - \theta_{s_2} |}{\area{T}} \\
& = &
a_2 \cdot \frac{\length{r} \cdot \Delta(s_1,s_2) }{\length{s_1}} \ .
\end{eqnarray*}

With Lemma \ref{lem-LengthEulerChara}, this implies
\begin{eqnarray*}
\norm{ r }  
& \geq & 
a_2 \cdot \frac{\length{r} \cdot \Delta(s_1,s_2) }{\length{s_1}} \\
& \geq &
a_2 \cdot \frac{\length{r} \cdot \Delta(s_1,s_2) }{6 ( -\chi_1 / b_1 ) } \\
& \geq &
\frac{a_2}{6} \cdot \frac{\Delta(s_1,s_2) }{-\chi_1 / b_1} \cdot \length{r} 
\end{eqnarray*}

Consequently, by the assumption that 
$
\Delta(s_1 , s_2) \geq 2 ( -\chi_1 /  b_1 )
$
and $a_2 \geq 2$, we obtain the desired inequality; 
$$
\norm{r} \geq 
\frac{2}{3} \, \length{r} \ .
$$

\end{proof}

We here give a proof of Theorem \ref{th-norm>length}. 

\begin{proof}[Proof of Theorem \ref{th-norm>length}]
Let $M$ be the exterior of a hyperbolic two-bridge knot or the exterior of a $(-2,3,n)$-pretzel knot with $n$ odd and at least $7$ in $S^3$. 

By virtue of Proposition~\ref{prop-Norm>Length}, it suffices to show that there exist two boundary slopes $s_1, s_2$ on $\partial M$ such that they are associated to some ideal points of $X(M)$, and some essential surfaces $S_1, S_2$ with boundary slopes $s_1, s_2$ satisfy 
$$
\Delta(s_1 , s_2) \geq 2 \frac{-\chi_i}{ b_i }, 
$$
where $\chi_i$ denote the Euler characteristic of $S_i$ and  
$b_i$ the number of boundary components of $S_i$ for $i=1,2$, respectively. 

First, consider the case where $M$ is the exterior of a hyperbolic two-bridge knot. 
Then it is proved in \cite{O} that every boundary slope on $\partial M$ is associated to an ideal point of norm curves in $X(M)$. 
See \cite{BC} for more explanations. 
In particular, as is stated in \cite[4.2]{IMS}, if the knot is twist knots, the character variety has only one norm curve $X_0$, and so $\norm{\cdot}=\norm{\cdot}_0$ holds. 
Also see \cite{BMZ} for this case. 
On the other hand, we know that any two-bridge knot is \textit{alternating}, i.e., it admits a diagram with alternatively arranged over-crossings and under-crossings running along it. 
Then, by \cite{DR}, both checkerboard surfaces $S_1$ and $S_2$ for a reduced alternating diagram of an alternating knot are essential. 

Let $C$ be the number of crossings of the alternating diagram, which is actually the minimal crossing number of the knot (\cite{Kau,Mu,Thi}). 
It is well-known that the distance $\Delta(s_1 , s_2)$ of the boundary slopes $s_1$ and $s_2$ of $S_1$ and $S_2$ is equal to $2C$. 
Also, we see that $b_i=1$ for the number of boundary components $b_i$ of $S_i$ and 
\[
\chi_1 + \chi_2 = 2 - C
\]
for the Euler characteristic $\chi_i < 0$ of $S_i$ with $i=1,2$. 
It follows that 
$$
\Delta(s_1 , s_2) = 2 C \geq 2C - 4 =  
2 \big( (- \chi_1) + (-\chi_2) \big) \geq 2 \frac{-\chi_i}{ b_i }, 
$$
for each $i=1,2$.

Next, consider the case where $M$ is the exterior of a $(-2,3,n)$-pretzel knot with $n$ odd and at least $7$ in $S^3$. 
These are well-known to be hyperbolic, and the following are proved in \cite{Ma}: 
The character variety $X(M)$ has only one norm curve $X_0$, and so $\norm{\cdot}=\norm{\cdot}_0$ holds.
Moreover, every boundary slope on $\partial M$ is associated to an ideal point of the norm curve $X_0$. 

On the other hand, by \cite[5.3.3. Corollaries]{IM}, there are two boundary slopes $s_1=16$ and $s_2=2n+6$ of two essential surfaces $S_1$ and $S_2$ with the number of boundary components $b_1=b_2=1$ and the Euler characteristics $\chi_1 =6-n$ and $\chi_2 = -1$, respectively. 
For these surfaces, we have the following as desired. 
\[
\Delta(s_1 , s_2) = (2n + 6) - 16 = 2n - 10 \ge  2 \frac{-\chi_1}{ b_1 } = 2 n - 12 \ge 2 \frac{-\chi_2}{ b_2 } =2 .
\]
\end{proof}

Also, as an application of Proposition~\ref{prop-Norm>Length}, we have the following. 
See \cite{CS04} for precise definitions of the terms used in the proposition. 

\begin{proposition}\label{prp-TwoSurfaceKnot}
Suppose that $M$ is the exterior of a hyperbolic knot in $S^3$ which contains only two strict essential surfaces, which are both spanning surfaces, up to isotopy. 
Then 
$$
\norm{r} \geq \frac{2}{3} \, \length{r}
$$
holds for any slope $r$ on $\partial M$ and for any horotorus $T$. 
\end{proposition}

\begin{proof}
The following statement is given as \cite[Corollary 7.6]{CS04}.  
Suppose that $M$ is a knot manifold (i.e., a compact, irreducible, orientable 3–manifold whose boundary is a torus) and that $M$ has at most two distinct isotopy classes of strict essential surfaces  
and $M$ is neither Seifert fibered nor an exceptional graph manifold.  
Let $S_1$ and $S_2$ be representatives of the two isotopy classes of connected strict essential surfaces.  
Let $s_i$ denote the boundary slope of $S_i$, $b_i$ denote the number of boundary components of $S_i$, and $\chi_i$ denote the Euler characteristic of $S_i$ for $i=1,2$.  
Then for $i = 1, 2$ we have  
\[
\chi_i \ge \frac{- b_1 b_2 \Delta (s_1, s_2) }{2}. 
\]  
When $S_1$ and $S_2$ are both spanning surfaces, i.e., $b_1=b_2=1$, it follows that  
\[
\Delta (s_1, s_2) \ge 2 \frac{-\chi_i}{b_i}
\]  
for $i=1,2$.  
Then, by Proposition~\ref{prop-Norm>Length},  
$$
\norm{r} \geq \frac{2}{3} \, \length{r}
$$  
holds for any slope $r$ on $\partial M$ and for any horotorus $T$.  
\end{proof}

In Proposition \ref{prop-Norm>Length}, the first condition that the two boundary slopes $s_1, s_2$ on $\partial M$ are associated to some ideal points of norm curves in $X(M)$ is necessary for our proof, but it is not easy to verify for a given knot.  
On the other hand, the second condition on the slopes $s_1, s_2$ can be satisfied for large classes of knots, as described below.

\begin{proposition}\label{prp-AltMontKnot}
Suppose that $M$ is the exterior of a hyperbolic alternating knot or a hyperbolic Montesinos knot in $S^3$.  
Then there exist two boundary slopes $s_1, s_2$ on $\partial M$ such that some essential surfaces $S_1, S_2$ with boundary slopes $s_1, s_2$ satisfy  
\[
\Delta(s_1 , s_2) \geq 2 \frac{-\chi_i}{ b_i }, 
\]
where $\chi_i$ denote the Euler characteristic of $S_i$ and  
$b_i$ the number of boundary components of $S_i$ for $i=1,2$, respectively.  
\end{proposition}

Here a knot having a diagram obtained by putting rational tangles in line is called a \textit{Montesinos knot}. 

\begin{proof}[Proof of Proposition~\ref{prp-AltMontKnot}]
Suppose first that $M$ is the exterior of a hyperbolic alternating knot.  
Then, as stated in the proof of Theorem~\ref{th-norm>length}, both checkerboard surfaces $S_1$ and $S_2$ for a reduced alternating diagram of an alternating knot are essential and satisfy $b_1 = b_2 = 1$, and  
\[
\Delta(s_1 , s_2) = 2 C \geq 2C - 4 = 
2 ( (- \chi_1) + (-\chi_2)) \ge 2 \frac{-\chi_i}{ b_i }, 
\]
where $C$ denotes the number of crossings of the alternating diagram and $\chi_i$ the Euler characteristic of $S_i$ for $i=1,2$.

Suppose next that $M$ is the exterior of a hyperbolic Montesinos knot $K$.  
Then, the following holds by \cite[Corollary 1.4]{IM2}.  
Among its non-meridional boundary slopes, let $s_1$ and $s_2$ be the maximum and minimum, respectively.  
Then, there exist two essential surfaces $S_1$ and $S_2$ with boundary slopes $s_1$ and $s_2$ such that  
\[
\Delta ( s_1 , s_2 ) \ge 2\,\left( \frac{-\chi_1}{b_1}  + \frac{-\chi_2}{b_2} \right) \, ,
\]
where $\chi_i$ and $b_i$ denote the Euler characteristic and the number of boundary components of $S_i$ for $i=1,2$.  
Thus, it follows that  
\[
\Delta(s_1 , s_2) \geq 2 \frac{-\chi_i}{ b_i }, 
\]
holds for $i=1,2$.  
\end{proof}

We can find an example of a $3$-manifold  
whose boundary has a slope $r$ with  
an arbitrarily large ratio $\norm{r} / \length{r}$.  

\begin{proposition}\label{prp-nonUnivBounds}  
There exist $3$-manifolds whose boundary tori have slopes $r$ with arbitrarily large ratio $\norm{r} / \length{r}$.  
\end{proposition}  

\begin{proof}  
For example, consider the exterior $M$ of a $(-2,3,n)$-pretzel knot in $S^3$ with $n$ odd, at least $7$, and $n \not\equiv 0 \pmod 3$.  
Then, by \cite{Ma}, the norm $\norm{m}$ of the meridional slope $m$ is equal to $3n - 9$.  
However, since Dehn filling along $m$ yields $S^3$, which is exceptional, the Euclidean length of the slope $m$ is at most 6 by \cite{Ag, L}.  
\end{proof}

\section{Lengths, norms and numerical slopes}\label{sec-Length>Norm}

The aim of this section is to give proofs of the following two propositions.

\begin{proposition}\label{prp-length}
Let $r_1 = p_1 / q_1$ and $r_2 = p_2 / q_2$ be finite numerical slopes on $\partial M$. 
Then 
$$
\frac{\length{r_1}}{q_1} + \frac{\length{r_2}}{q_2} > | r_1 - r_2 |
$$
holds for the maximal horotorus $T$. 
\end{proposition}

\begin{proposition}\label{prp-norm}
Let $r_1 = p_1 / q_1$ and $r_2 = p_2 / q_2$ be finite numerical slopes on $\partial M$. 
Then 
\[
\frac{\norm{r_1}}{q_1 \norm{m}} + \frac{\norm{r_2}}{q_2 \norm{m}} \geq | r_1 - r_2 |
\]
holds.
Moreover, if $r_1$ is greater than or equal to the maximal boundary slope and 
$r_2$ is less than or equal to the minimal boundary slope for $M$, 
then equality holds. 
\end{proposition}

These propositions imply the next theorem immediately.

\begin{theorem}\label{th-length>diff>norm}
Let $r_1 = p_1 / q_1$ and $r_2 = p_2 / q_2$ be finite numerical slopes on $\partial M$. 
If $r_1$ is greater than or equal to the maximal boundary slope and 
$r_2$ is less than or equal to the minimal boundary slope for $M$, then 
$$
\frac{\length{r_1}}{q_1} + \frac{\length{r_2}}{q_2} > | r_1 - r_2| 
= \frac{\norm{r_1}}{q_1 \norm{m}} + \frac{\norm{r_2}}{q_2 \norm{m}} 
$$
holds for the maximal horotorus $T$. 
\end{theorem}

Restricting the slopes to integral ones, we obtain Theorem~\ref{th-length>norm} from this.

We remark that the following is an immediate corollary of Proposition~\ref{prp-length} together with Lemma~\ref{lem-LengthEulerChara}.

\begin{corollary}
Let $r_1 = p_1 / q_1$ and $r_2 = p_2 / q_2$ be finite numerical slopes on $\partial M$. 
Suppose that there exist essential surfaces $S_1$ and $S_2$ with boundary slopes $r_1$ and $r_2$. 
Let $\chi_i$ denote the Euler characteristic of $S_i$ and $b_i$ the number of boundary components of $S_i$ for $i=1,2$. 
Then the following hold:
\[
6 \left( \frac{ -\chi_1}{b_1 q_1} + \frac{ -\chi_2}{b_2 q_2}  \right) > | r_1 - r_2|,
\]
\[
6 \left( q_2 \frac{ -\chi_1}{b_1} + q_1 \frac{ -\chi_2}{b_2 }  \right) > \Delta(  r_1 , r_2 ).
\]
\end{corollary}

Note that these are extensions of the results previously given for Montesinos knots in \cite{IM}.

\subsection{Lengths of numerical slopes}\label{subsec-Length>Norm-Length}

The idea of the proof of Proposition \ref{prp-length} was developed in \cite{I01}. 
Similar technique is also used in \cite{I08}. 


\begin{proof}[Proof of Proposition \ref{prp-length}]
Let $r_1 = p_1 / q_1$ and $r_2 = p_2 / q_2$ be non-meridional numerical slopes on $\partial M$.
We choose a horotorus $T$ and identify $\partial M$ with $T$. 
Let $\widetilde{T}$ be a component of the preimage of $T$ in the universal cover of Int$(M)$. 
This $\widetilde{T}$ is naturally identified with the Euclidean plane $\mathbb{E}^2$. 
The preimage of a point on $T$ gives a lattice on $\widetilde{T}$ with base point $O$ as utilized in \cite{I01}. 
Then, if a primitive lattice point $X$ corresponds to a slope $r$ on $T$, then the distance $d(O,X)$ between $O$ and $X$ equals the length $\length{r}$ of the slope $r$. 

Now, take the lattice points $P$ and $Q$ corresponding to $r_1$ and $r_2$ on $\widetilde{T}$. 
%
Let $P'$ (resp. $Q'$) be the intersection points between $OP$ (resp. $OQ$) and the line which contains all the points corresponding to the integral slopes. 
Then, we see that $d(O,P) = q_1 d (O,P')$ and $d(O,Q)= q_2 d (O,Q')$. 
Also note that the distance $d(P',Q')$ between $P'$ and $Q'$ is equal to $|r_1 - r_2| \,\length{m}$ for the meridian $m$.
Together with the triangle inequality $d(P',Q') < d (O,P') + d(O,Q')$ and Lemma~\ref{lem-minimallength}, it implies that
\[
\frac{\length{r_1}}{q_1} + \frac{\length{r_2}}{q_2} > | r_1 - r_2 |\, \length{m} \ge | r_1 - r_2 | .
\]
\end{proof}

\subsection{Norms of numerical slopes}

\begin{proof}[Proof of Proposition \ref{prp-norm}]
Let $r_1 = p_1 / q_1$ and $r_2 = p_2 / q_2$ be finite numerical slopes on $\partial M$, and $s_1=t_1/u_1, s_2=t_2/u_2, \cdots,s_n=t_n/u_n$ be the all boundary slopes on $\partial M$ with $u_i >0$. 

Then, we have the following from Lemma~\ref{lem-NormDistance}. 
\begin{align*}
\frac{\norm{p_1/q_1}}{q_1} + \frac{\norm{p_2/q_2}}{q_2} 
&=
\frac{1}{q_1} \sum^n_{i=1} a_i \Delta ( p_1/q_1 , s_i) 
+ \frac{1}{q_2}\sum^n_{i=1} a_i \Delta ( p_2/q_2 , s_i) \\
&= \frac{1}{q_1} \sum^n_{i=1} a_i | p_1 u_i - q_1 t_i| + \frac{1}{q_2} \sum^n_{i=1} a_i  | p_2 u_i - q_2 t_i | \\
&= \sum^n_{i=1} a_i u_i \left( \left| \frac{p_1}{q_1} - \frac{t_i}{u_i} \right| + \left| \frac{p_2}{q_2} - \frac{t_i}{u_i} \right| \right) \\
&\ge \left( \sum^n_{i=1} a_i u_i \right) \left| \frac{p_1}{q_1} - \frac{p_2}{q_2} \right| \\
&= \left( \sum^n_{i=1} a_i \Delta ( 1/0 , t_i/u_i ) \right) \left| \frac{p_1}{q_1} - \frac{p_2}{q_2} \right| \\
&= \norm{m} \cdot | r_1 - r_2|
\end{align*}
where $a_i$ is a non-negative even integer for each $i$ given in Lemma~\ref{lem-NormDistance}. 

It then follow that 
\[
\frac{\norm{r_1}}{q_1 \norm{m}} + \frac{\norm{r_2}}{q_2 \norm{m}} \geq | r_1 - r_2|.
\]
Moreover, if $r_1$ is greater than or equal to the maximal boundary slope and $r_2$ is less than or equal to the minimal boundary slope for $M$, then 
\[
\left| \frac{p_1}{q_1} - \frac{t_i}{u_i} \right| + \left| \frac{p_2}{q_2} - \frac{t_i}{u_i} \right| 
= \left| \frac{p_1}{q_1} - \frac{p_2}{q_2} \right| 
\]
holds for any $i$, implying that the equality in the above holds. 
\end{proof}

\section{Boundary slope diameter}\label{sec-diam} 

\begin{proof}[Proof of Theorem~\ref{th-diam}]

Let $s_1=t_1/u_1, \cdots,s_n=t_n/u_n$ be the all boundary slopes on $\partial M$ with $u_i >0$, and let $r=p/q$ be a numerical boundary slope on $\partial M$. 
Then, we have the following. 
\begin{align*}
\norm{r}=\norm{p/q} 
&= \sum^n_{i=1} a_i \Delta ( p/q , t_i/u_i ) \\
&= \sum^n_{i=1} a_i | p u_i - t_i q | \\
&= \sum^n_{i=1} a_i q u_i \left| \frac{p}{q} - \frac{t_i}{u_i}  \right| \\
&\le q \left( \sum^n_{i=1} a_i u_i \right) \diam{\BoundarySlopeSet} \\
&= q \,\norm{m} \,\diam{\BoundarySlopeSet} 
\end{align*}
with $a_i \ge 0$ by Lemma~\ref{lem-NormDistance}. 
It follows that
$$
\diam{\BoundarySlopeSet} \geq \frac{\norm{r}}{q \norm{m}}.
$$

Next, we show that the equality above cannot hold. 
Suppose that the equality holds. 
Since all the differences $| r - s_i|$ are at most $\diam{\BoundarySlopeSet}$, if the equality holds, all but at most one of $a_i$'s must be 0. 
This contradicts that $\norm{ \cdot }$ is a (genuine) norm on $H_1( \partial M ; \mathbb{R})$. 
\end{proof}

We also have an upper bound on $\diam{\BoundarySlopeSet}$ in terms of the Culler-Shalen norms of some boundary slopes as a corollary of Proposition~\ref{prp-norm}. 

\begin{corollary}\label{cor-ubdiam}
Let $s_1, s_2, \dots, s_n$ be numerical boundary slopes on $\partial M$ such that $s_1 > s_2 > \cdots > s_n$. 
Then
$$
\frac{\norm{s_1}}{\norm{m} \cdot \Delta (s_1,m)} 
+ \frac{\norm{s_n}}{\norm{m} \cdot \Delta (s_n,m)} 
\geq \diam{\BoundarySlopeSet} 
$$
holds. 
It follows that 
$$
2 \max_{1 \leq i \leq n} \left\{ \frac{\norm{s_i}}{\norm{m} \cdot \Delta (s_i,m)} \right\}
\geq \diam{\BoundarySlopeSet} .
$$    
\end{corollary}

\begin{proof}
The first assertion immediately follows from Proposition~\ref{prp-norm}. 
It then implies that the second assertion. 
\end{proof}

\section{Slopes for the figure-eight knot}\label{sec-FigureEight}

In this section, we give a description of the Euclidean length on the maximal horotorus and the Culler-Shalen norm of slopes for the figure-eight knot exterior. 
This demonstrates how our inequalities are effective.

Let $M$ be the exterior of the figure-eight knot in $S^3$. This knot is well-known to be hyperbolic. 
See \cite{Th} for example. We take the maximal horotorus $T$. 
In \cite{Th}, the modulus of the maximal horotorus is investigated in detail. On the universal cover of $T$, we have the following diagram.

\begin{figure}[htb]
\includegraphics[width=0.7\textwidth]{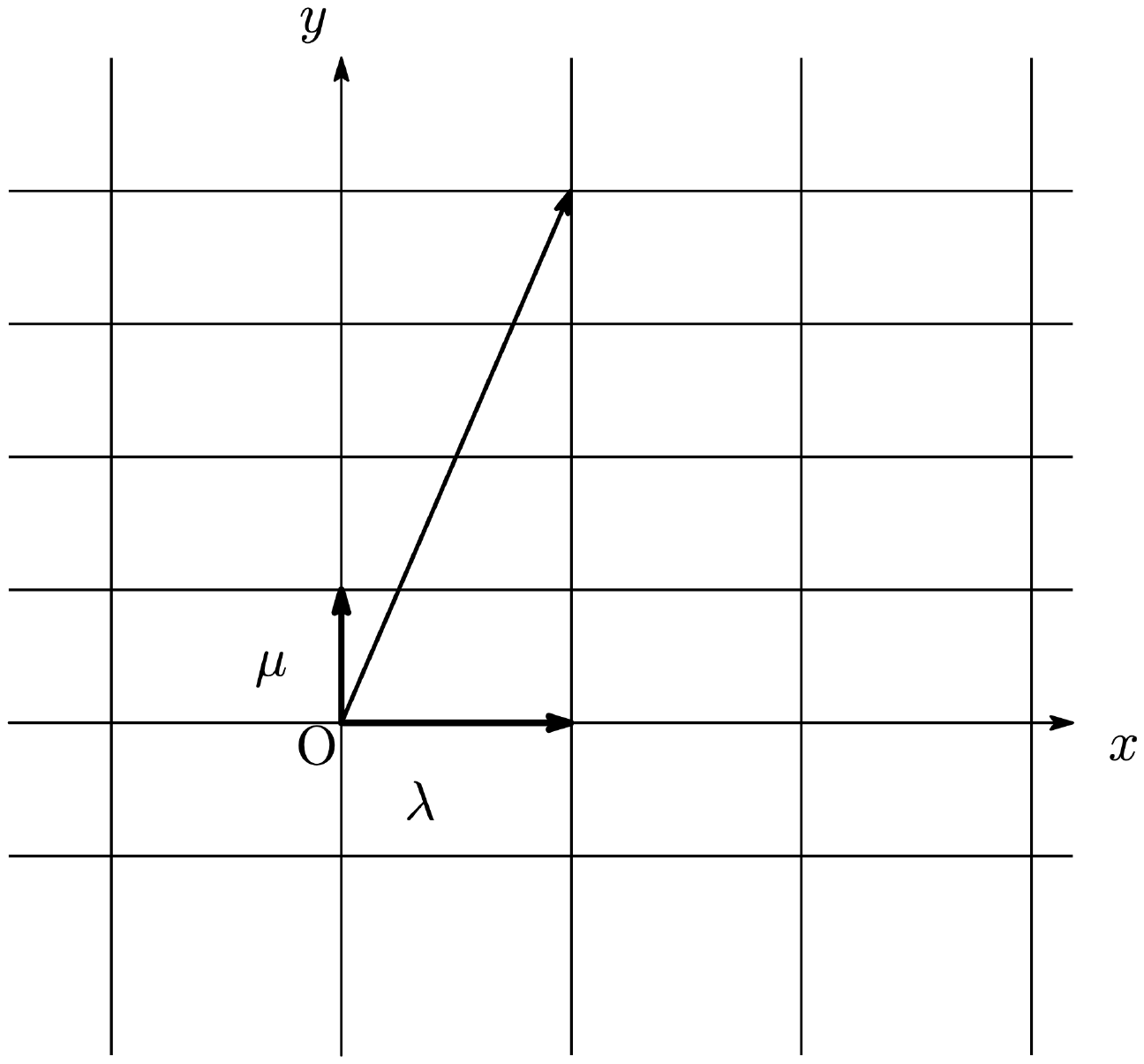}
\caption{}
\end{figure}

In the figure, $\mu, \lambda$ denote the lifts of the meridian $m$ and the preferred longitude, with suitable orientations.

It is shown that $\mu$ and $\lambda$ are perpendicular to each other, and $\length{\mu} = 1$ and $\length{\lambda} = 2 \sqrt{3}$. Thus, for example, it is calculated that $\length{r_{(1,4)}} = 2 \sqrt{7}$ for the slope $r_{(1,4)}$, which runs once in the longitudinal direction and 4 times in the meridional direction.

In general, $\length{p/q}$ of the numerical slope $p/q$ is calculated as
\[
\length{p/q} = \sqrt{ p^2 + 12q^2} = q \sqrt{(p/q)^2 + 12}.
\]

On the other hand, 
%
%
for the figure-eight knot exterior case, from the results given in \cite{O}, we have 
$\norm{r} = 2 \, \Delta (r, 4/1) + 2 \, \Delta (r, -4/1) $. 
It follows that the Culler-Shalen norm $\norm{p/q}$ of the numerical slope $p/q$ is calculated as
\begin{align*}
\norm{p/q} &= 2 \, \Delta (p/q, 4/1) + 2 \, \Delta (p/q, -4/1) \\
&= 2 q ( |p/q - 4| + | p/q + 4 | )   
\end{align*}

%
%

Thus, by considering the ratio $\norm{r}/\length{r}$ for slopes $r$, we obtain
$$
\frac{8}{\sqrt{7}} \length{r} \leq \norm{r} \leq \frac{8}{\sqrt{3}} \length{r} \ .
$$
Please compare this with the inequality given by Theorem~\ref{th-norm>length}. 
Moreover, in this case, $\norm{m} = 4$,  
so we have  
$$
\length{r} \geq \frac{\sqrt{3}}{2} \, \frac{\norm{r}}{\norm{m}} \ .
$$
The equality holds for the preferred longitude.

In particular, for an integral slope $r$ with $r \ge 4$ or $r \le -4$, we have 
$$
\length{r} > \frac{\norm{r}}{\norm{m}} \ .
$$
Please compare this with the inequality given by Theorem~\ref{th-length>norm}. 
This shows that the inequality given in the theorem is optimal in a sense. 

Concerning Theorem~\ref{th-diam} and Corollary~\ref{cor-ubdiam}, we see that $\diam{\BoundarySlopeSet} = 4-(-4) = 8$, $\norm{4/1} = \norm{-4/1} = 16$, and $\norm{m}=4$. 
Thus we have
\[
\diam{\BoundarySlopeSet} = 8 > \frac{\norm{4/1}}{\norm{m}} = \frac{\norm{-4/1}}{\norm{m}} = 4 
\]
and 
\[
\frac{\norm{4/1}}{\norm{m}} + \frac{\norm{-4/1}}{\norm{m}} 
= 8 = 
\diam{\BoundarySlopeSet} .
\]

\bibliographystyle{amsplain}
\bibliography{references}

\providecommand{\bysame}{\leavevmode\hbox to3em{\hrulefill}\thinspace}
\providecommand{\MR}{\relax\ifhmode\unskip\space\fi MR }
\providecommand{\MRhref}[2]{%
  \href{http://www.ams.org/mathscinet-getitem?mr=#1}{#2}
}
\providecommand{\href}[2]{#2}
\begin{thebibliography}{10}

\bibitem{Ad}
Colin~C. Adams, \emph{Waist size for cusps in hyperbolic 3-manifolds}, Topology \textbf{41} (2002), no.~2, 257--270. \MR{1876890}

\bibitem{Ag}
Ian Agol, \emph{Bounds on exceptional {D}ehn filling}, Geom. Topol. \textbf{4} (2000), 431--449. \MR{1799796}

\bibitem{BH}
Steven~A. Bleiler and Craig~D. Hodgson, \emph{Spherical space forms and {D}ehn filling}, Topology \textbf{35} (1996), no.~3, 809--833. \MR{1396779}

\bibitem{BC}
Hans~U. Boden and Cynthia~L. Curtis, \emph{The {$SL(2,\Bbb C)$} {C}asson invariant for {D}ehn surgeries on two-bridge knots}, Algebr. Geom. Topol. \textbf{12} (2012), no.~4, 2095--2126. \MR{3020202}

\bibitem{BMZ}
S.~Boyer, T.~Mattman, and X.~Zhang, \emph{The fundamental polygons of twist knots and the {$(-2,3,7)$} pretzel knot}, K{NOTS} '96 ({T}okyo), World Sci. Publ., River Edge, NJ, 1997, pp.~159--172. \MR{1664959}

\bibitem{BZ}
S.~Boyer and X.~Zhang, \emph{Finite {D}ehn surgery on knots}, J. Amer. Math. Soc. \textbf{9} (1996), no.~4, 1005--1050. \MR{1333293}

\bibitem{B}
Steven Boyer, \emph{Dehn surgery on knots}, Handbook of geometric topology, North-Holland, Amsterdam, 2002, pp.~165--218. \MR{1886670}

\bibitem{CGLS}
Marc Culler, C.~McA. Gordon, J.~Luecke, and Peter~B. Shalen, \emph{Dehn surgery on knots}, Ann. of Math. (2) \textbf{125} (1987), no.~2, 237--300. \MR{881270}

\bibitem{CS99}
Marc Culler and Peter~B. Shalen, \emph{Boundary slopes of knots}, Comment. Math. Helv. \textbf{74} (1999), no.~4, 530--547. \MR{1730656}

\bibitem{CS04}
\bysame, \emph{Knots with only two strict essential surfaces}, Proceedings of the {C}asson {F}est, Geom. Topol. Monogr., vol.~7, Geom. Topol. Publ., Coventry, 2004, pp.~335--430. \MR{2172490}

\bibitem{DR}
Charles Delman and Rachel Roberts, \emph{Alternating knots satisfy {S}trong {P}roperty {P}}, Comment. Math. Helv. \textbf{74} (1999), no.~3, 376--397. \MR{1710698}

\bibitem{I01}
Kazuhiro Ichihara, \emph{Exceptional surgeries and genera of knots}, Proc. Japan Acad. Ser. A Math. Sci. \textbf{77} (2001), no.~4, 66--67. \MR{1829376}

\bibitem{I08}
\bysame, \emph{Integral non-hyperbolike surgeries}, J. Knot Theory Ramifications \textbf{17} (2008), no.~3, 257--261. \MR{2400664}

\bibitem{IM}
Kazuhiro Ichihara and Shigeru Mizushima, \emph{Crossing number and diameter of boundary slope set of {M}ontesinos knot}, Comm. Anal. Geom. \textbf{16} (2008), no.~3, 565--589. \MR{2429969}

\bibitem{IM2}
\bysame, \emph{Lower bounds on boundary slope diameters for {M}ontesinos knots}, Kyungpook Math. J. \textbf{49} (2009), no.~2, 321--348. \MR{2554890}

\bibitem{IMS}
Masaharu Ishikawa, Thomas~W. Mattman, and Koya Shimokawa, \emph{Exceptional surgery and boundary slopes}, Osaka J. Math. \textbf{43} (2006), no.~4, 807--821. \MR{2303551}

\bibitem{Kau}
Louis~H. Kauffman, \emph{State models and the {J}ones polynomial}, Topology \textbf{26} (1987), no.~3, 395--407. \MR{899057}

\bibitem{L}
Marc Lackenby, \emph{Word hyperbolic {D}ehn surgery}, Invent. Math. \textbf{140} (2000), no.~2, 243--282. \MR{1756996}

\bibitem{Ma}
Thomas~W. Mattman, \emph{The {C}uller-{S}halen seminorms of the {$(-2,3,n)$} pretzel knot}, J. Knot Theory Ramifications \textbf{11} (2002), no.~8, 1251--1289. \MR{1949779}

\bibitem{Me84}
W.~Menasco, \emph{Closed incompressible surfaces in alternating knot and link complements}, Topology \textbf{23} (1984), no.~1, 37--44. \MR{721450}

\bibitem{Mu}
Kunio Murasugi, \emph{Jones polynomials and classical conjectures in knot theory}, Topology \textbf{26} (1987), no.~2, 187--194. \MR{895570}

\bibitem{O}
Tomotada Ohtsuki, \emph{Ideal points and incompressible surfaces in two-bridge knot complements}, J. Math. Soc. Japan \textbf{46} (1994), no.~1, 51--87. \MR{1248091}

\bibitem{R}
Dale Rolfsen, \emph{Knots and links}, Mathematics Lecture Series, vol. No. 7, Publish or Perish, Inc., Berkeley, CA, 1976. \MR{515288}

\bibitem{S}
Peter~B. Shalen, \emph{Representations of 3-manifold groups}, Handbook of geometric topology, North-Holland, Amsterdam, 2002, pp.~955--1044. \MR{1886685}

\bibitem{Thi}
Morwen~B. Thistlethwaite, \emph{A spanning tree expansion of the {J}ones polynomial}, Topology \textbf{26} (1987), no.~3, 297--309. \MR{899051}

\bibitem{Th}
W.P. Thurston, \emph{The geometry and topology of three-manifolds}, Notes, Princeton University, 1980.

\end{thebibliography}

\end{document}